\documentclass[12pt,a4paper,oneside]{amsart}
\usepackage[utf8]{inputenc}

\usepackage{amsmath,amscd,hyperref, amsfonts, amssymb, amsthm,}
\usepackage{setspace,kantlipsum}
\usepackage{tikz-cd}
\usepackage{mathrsfs}
\usepackage{textcomp}
\usepackage[margin=1in]{geometry}
\usepackage{colonequals}
\usepackage{mathtools}
\usepackage{stackrel}
\usepackage[all]{xy} 
\usepackage{youngtab} 
\usepackage{young} 
\usepackage{ytableau}

\usepackage[skins]{tcolorbox}
\usepackage{mathrsfs}
\usepackage{xcolor}
\usepackage{enumitem}
\usepackage {amssymb,color}
\usepackage{verbatim}
\usepackage[bb=dsserif]{mathalpha}
\usepackage{bm}

\usepackage[toc]{appendix}

\theoremstyle{plain} 
\newtheorem{prop}{Proposition}[section]

\newtheorem{lemma}[prop]{Lemma}
\newtheorem{theorem}[prop]{Theorem}
\newtheorem{conjecture}[prop]{Conjecture}

\newtheorem{remark}[prop]{Remark}

\newtheorem{definition}[prop]{Definition}
\newtheorem{corollary}[prop]{Corollary}

\newcommand{\C}{\mathbb{C}}

\newcommand{\Z}{\mathbb{Z}}

\newcommand{\R}{\mathbb{R}}

\newcommand{\D}{\mathcal{D}}

\newcommand{\M}{\mathcal{M}}
\renewcommand{\S}{\mathcal{S}}

\newcommand{\gr}{\mathrm{gr}}

\newcommand{\CC}{\operatorname{CC}}

\newcommand{\ol}[1]{\overline{#1}}

\newcommand{\fk}{\textmd{k}}
\newcommand{\G}{\textmd{G}}
\newcommand{\fgv}{\mathfrak{g}^{\vee}}
\newcommand{\ftv}{\mathfrak{t}^{\vee}}
\newcommand{\Tv}{\textmd{T}^{\vee}}
\newcommand{\fts}{\S_{Ft}}

\newcommand{\lra}{\longrightarrow}

\begin{document}
\title[]{SOME UNIPOTENT ARTHUR PACKETS  FOR P-ADIC  SPLIT  $F4$}
\author{Leticia Barchini}
\address{Department of Mathematics, Oklahoma State University, Stillwater, OK 74078}
\email{leticia@math.okstate.edu}

\author{Andr\'as C. L\H{o}rincz}
\address{Department of Mathematics, University of Oklahoma, Norman, OK 73019}
\email{lorincz@ou.edu}

\maketitle

{\centering\footnotesize \em  Dedicated to Wilfried Schmid on the occasion of his 80th birthday.\par}

\begin{abstract}
Let $\G(\textmd{k})$ be the split form of the simple exceptional $p$-adic group of type $F_4,$  and let
$\mathcal O = F_4(a_3)$ be the minimal distinguished nilpotent orbit. Our main result concerns the class of unipotent
representations with cuspidal support at infinitesimal character $\Lambda$ determined by $\mathcal O.$
These representations  are parameterized by local systems, $\{(\S, \mathcal L)\}.$ 
We compute the characteristic cycles of the perverse sheaves $\text{IC}(\S, \mathcal L)$ 
and determine all micro-packets in the sense of \cite{Vo93}. 
In \cite{cmbo}, the authors  introduced a notion of weak Arthur packets in the p-adic setting. They conjectured that
 weak Arthur packets are unions of Arthur packets, in an appropriate sense.
We verify that weak Arthur packets are unions of micro-packets.
\end{abstract}

\section{Introduction}

Let $\textmd{k}$ be a nonarchimedean local field of characteristic $0$ with residue field $\textmd{F}_q$
of cardinality $q.$ Let  $\textmd{G}$ be a  reductive algebraic group defined over $\textmd{k}.$   
Let
$\textmd{G}(\fk)$ stand for the group of $\fk$-rational points. We assume throughout that the inner class of $\textmd{G}(\fk)$ includes
the split form.
Our work relates to $\Pi^{\text{Lus}}(\textmd{G}(\fk));$ the equivalence classes of irreducible 
$\textmd{G}(\fk)$-representations with unipotent cuspidal support defined in \cite{Lus95}.
\vskip .1in

Let $\G^{\vee}$ be the complex Langlands' dual group associated to $\G$,
and let  $\Tv \subset \G^{\vee}$ be a maximal torus. Write $\Tv = \Tv_c \; \Tv_{r},$ the polar decomposition of $\Tv$ where
$\Tv_c$ is the maximal compact subgroup of $\Tv$.
We write $\fgv, \ftv, \ftv_{r}$ for the Lie algebras of
$\G^{\vee}, \Tv,  \Tv_{r}.$ For a nilpotent orbit $\mathcal O \subset \fgv$ we choose a
$\mathfrak{sl}_2$-triple $\{{\it e}, {\it h}, {\it f}\}$ with ${\it h} \in \ftv_r.$
The element $q^{\frac{1}{2}\it{h}} \in \G^{\vee}$ is semisimple. We consider
\[\Pi_{q^{\frac{1}{2}\it{h}}}^{\text{Lus}}(\textmd{G}(\fk))\]
the equivalence classes of  irreducible  unipotent representations with cuspidal support  with infinitesimal character determined by 
 $q^{\frac{1}{2}\it{h}}.$ 
 \vskip .2in

 In analogy  with the case of  reductive groups defined over $\R$, 
in \cite{cmbo}, the authors introduced the  notion of  the  p-adic  {\it weak Arthur packet} with infinitesimal character $q^{\frac{1}{2}\it{h}}.$
They identified a {\it basic Arthur packet} contained in the  {\it weak Arthur packet}.  They
conjectured that 
 {\it weak Arthur packets}  are  union of Arthur packets, in an appropriate sense. See Conjecture \ref{cj2}.
 
 \vskip .2in
 
 In this paper, heavily relying on the work by Dan Ciubotaru in \cite{ci08} and \cite{ci22}, we verify Conjecture \ref{cj2}
when $\G^{\vee} = F_4$ and $\mathcal O$ is the minimal distinguished nilpotent orbit of type $F_4$, $F_4(a_3)$. We use the micro-local approach of
\cite{Vo93}. We show  that the {\it basic Arthur packet} $\Pi_{F_4(a_3)}^{\text{Art}}(\G(\fk))$ is a micro-packet and we exhibit the weak Arthur packet $\Pi^{\text{weak}}_{F_4(a_3)}(\G(\fk))$
as a union of micro-packets. 
\vskip .2in

Weak Arthur packets  admit an explicit description  in
terms of $AZ$, the normalized Aubert-Zelevinky involution \cite[Definition 1.5, Corollary 3.9]{Au95}.
 It is known that  $AZ$  preserves $ \Pi_{q^{\frac{1}{2}\it{h}}}^{\text{Lus}}(\textmd{G}(\fk)).$
  The  irreducible representations in $\Pi_{q^{\frac{1}{2}\it{h}}}^{\text{Lus}}(\textmd{G}(\fk))$ are parameterized by geometric data, that we call {\it geometric Langlands' parameters.} It is expected that the geometric Langlands' parameters
of $X$ and 
$AZ(X)$ are related  via the Fourier transform, see Conjecture \ref{cj1}.
In section \ref{last},  assuming that Conjecture \ref{cj1} holds true,  we show, in general, that 
weak Arthur packets are a union of micro-packets.

\vskip .2in

Here is an outline of the paper. 
 Subsections 2.1 to 2.4  contain brief descriptions of background material needed to state Conjecture \ref{cj2}.
 Subsection 2.5 concerns Conjecture \ref{cj2}. Subsection 2.6 describes the techniques used to compute micro-packets.
 A detailed description of $\Pi^{\text{weak}}_{F_4(a_3)}(\G(\fk))$ is presented in Section 3. Technical aspects of the
 computations are included in the Appendix.
  Section 4  relates Conjecture \ref{cj1} and Conjecture \ref{cj2}.

\vskip .2in
\noindent{\bf Acknowledgements.} 
We thank Dan Ciubotaru for his detailed and  helpful answers to our questions regarding \cite{cmbo}.
 We also thank Peter Trapa for his continuous and generous support. 


\section{Preliminaries}
  
\subsection{Geometric Langlands Parameters}\label{lp}

Let $\textmd{k}$ be a nonarchimedean local field of characteristic $0$ with residue field $\textmd{F}_q$
of cardinality $q.$ Let $W_{\fk}$ denote the Weil group  with Inertia subgroup  $I_{\fk}$ and norm $||\; ||.$
Fix $Fr,$  a Frobenius element that generates $W_{\fk}/I_{\fk},$ \cite[(1.4)]{ta}.
Denote by $W'_{\fk}$  the  Weil-Deligne group.  That is,
  \[W'_{\fk} = W_{\fk}\ltimes \C,\]
where  $W_{\fk}$ acts on $\C$ by
 $\omega \cdot z = ||\omega|| \; z$ with  $\omega \in W_{\fk}$ and $z \in \C.$
 \vskip .1in
 
 A Langlands' parameter is a continuous group homomorphism 
 \[\phi: W'_{\fk}\lra \G^{\vee}\]
 satisfying certain compatibility conditions (for example, see \cite{Vo93}).  Attached to a Langlands' parameter $\phi$ satisfying
   $\phi(I_{\fk}) = \mathbb{1}$ 
 and $\phi(Fr) = q^{\frac{1}{2}\it{h}},$ there is a set of irreducible unipotent representations  with cuspidal support and infinitesimal character
 $ q^{\frac{1}{2}\it{h}},$  
  $\Pi_{\phi}^{\text{Lang}}(\textmd{G}(\fk)),$ such that
  
  \[ \Pi_{q^{\frac{1}{2}\it{h}}}^{\text{Lus}}(\textmd{G}(\fk)) = \underset{\substack{ \{\phi: \phi(I_{\fk})= \mathbb{1}\\ \qquad \phi(Fr) = q^{\frac{1}{2}\it{h}}  \}  }}{\bigcup}  \Pi_{\phi}^{\text{Lang}}(\textmd{G}(\fk)).\]
  
  The classification of irreducible modules in $\Pi_{q^{\frac{1}{2}\it{h}}}^{\text{Lus}}(\textmd{G}(\fk))$
  is expressed in terms of the spaces

 \[ \G^{\vee}({\it h}) = Z_{\G^{\vee}}({\it h});\;
 \fgv(2)  = \{ {\it x}\in \fgv : [{\it h}, {\it x}] =  2 \; {\it x}\};\;
  \fgv(- 2)   = \{ {\it x}\in \fgv : [{\it h}, {\it x}] = - 2 \;  {\it x}\}.
   \]
   
   \vskip .1in
    
  The group $\G^{\vee}({\it h})$ acts with finitely many orbits  on both  $\fgv(2)$ and  $\fgv(-2).$
  It is known that the  set of equivalence classes of Langlands' parameters with infinitesimal character $ q^{\frac{1}{2}\it{h}^{\vee}}$
  is in bijection with the set of  $\G^{\vee}({\it h})$-orbits on $\fgv(2)$  (for example, see \cite{Vo93}).
  We write $\S_{\phi} \subset \fgv(2)$ for the $\G^{\vee}({\it h})$-orbit that corresponds to the equivalence class $[\phi].$
  
  \vskip .2in
  
  Let $\text{Per}_{\G^{\vee}(\it{h})}(\fgv(2))$ denote the category of $\G^{\vee}(\it{h})$-equivariant perverse sheaves. The simple objects of this category are the intersection cohomology complexes 
  \[ \{IC(\S, \mathcal L): \mathcal L \text{ is a  simple equivariant local system on  the $\G^{\vee}(\it{h})$-orbit }
  \S \subset \fgv(2) \}.\]
  
\begin{theorem} (Deligne-Langlands-Lusztig)\label{dll}  
There is a bijection 
\begin{align*}
\text{Irr } (\text{Per}_{\G^{\vee}(\it{h})}(\fgv(2)) &\longleftrightarrow \Pi_{q^{\frac{1}{2}\it{h}}}^{\text{Lus}}(\textmd{G}(\fk)\\
IC(\S, \mathcal L ) & \mapsto X:= X(q^{\frac{1}{2}\it{h}}, S, \mathcal L).
\end{align*}
\end{theorem}
\vskip.1in

If $X= X(q^{\frac{1}{2}\it{h}}, \S, \mathcal L),$ we say that the pair $(\S, \mathcal L)$
is the {\it geometric Langlands' parameter of } X.

\vskip .2in

\subsection{Orbits duality}\label{od}

Identify $\fgv(2)^*$ with $\fgv(-2)$ by using a fixed non-degenerate invariant symmetric form on $\fgv.$
\vskip .2in

\begin{lemma}\cite[Corollary 2]{Pja75}
For each $\G^{\vee}(\it{h})$-orbit $\S \subset \fgv(2)$, there exists a unique $\G^{\vee}(\it{h})$-orbit
$\S^{Pt} \subset \fgv(-2)$ so that the $\G^{\vee}(\it{h})$-equivariant isomorphism
\begin{align*}
T^*(\fgv(2)) & \lra T^*(\fgv(-2))\\
(x,\xi) & \mapsto (\xi, x)
\end{align*}
restricts to an isomorphism of co-normal bundles
\[ \overline{T^*_{\S}(\fgv(2)) } \simeq \overline{T^*_{\S^{Pt}}(\fgv(-2))}.\]
\end{lemma}

The assignment 

\begin{equation}\label{Pt}
\S\mapsto \S^{Pt} \text{ is a bijection of orbits.}
\end{equation}

\vskip .3in

There is a second canonical bijection between 
$\G^{\vee}({\it h})$-orbits on  $\fgv (2)$ and  those on $\fgv (-2).$
If $\S = \G^{\vee}({\it h}) \cdot {\it x} \subset \fgv(2),$ then there exists a map
\[di: \mathfrak{sl}_2 \lra \fgv \text {such that }\]
$di\begin{pmatrix}  0 & 1\\0 & 0\end{pmatrix}       = {\it x} \in \fgv(2);    
di\begin{pmatrix}  0& 0 \\1 & 0\end{pmatrix} = {\it y}\in \fgv(-2); \text{ and }
di\begin{pmatrix}  1& 0\\0 & -1\end{pmatrix} =\it{h}.$
\vskip .1in

The assignment
\begin{equation}\label{trans}
\S= \G^{\vee}({\it h}) \cdot {\it x} \mapsto  \; ^t\S = \G^{\vee}(\it{h}) \cdot{ \it y}   \text{ is a bijection of orbits.}
\end{equation}

 One can check that  the component groups $A(\S)= Z_{\G^{\vee}(\it{h})}(\it{x})/Z_{\G^{\vee}(\it{h})}(\it{x})_0$
and $A(^t\S)$ are isomorphic. We denote by $\mathcal{L} \mapsto^t\mathcal{L}$ the corresponding correspondence on equivariant local systems.
\vskip .2in

 We conclude that the map

\begin{equation}\label{hat}
\S \mapsto \widehat{\S} = ^t\S^{Pt}
\end{equation}
is an involution on the set of $\G^{\vee}({\it h})$-orbits on $\fgv(2).$  Compare to  \cite[page 38]{cfmmx}
\vskip .2in

\subsection{Fourier Trasform and Aubert-Zelevinski duality}\label{FT}
We  denote  by Ft the Fourier transform from  $\text{Per}_{\G^{\vee}(\it{h})}(\fgv(2))$ to
$\text{Per}_{\G^{\vee}(\it{h})}(\fgv(-2))$  in the sense of \cite{ks}, \cite[Section 3.2]{em}.
Then, for a equivariant local system $(\fts, \mathcal L_{Ft})$ with $ \fts \subset \fgv(-2),$ we have 
 
\[ Ft\big(IC(\S, \mathcal L)\big)  =  IC(\fts, \mathcal L_{Ft}).\]

\begin{conjecture}\label{cj1}  Assume $X = X(q^{\frac{1}{2}\it{h}}, \S, \mathcal{L})$
is an irreducible unipotent representation in $\Pi_{q^{\frac{1}{2}\it{h}}}^{\text{Lus}}(\G(\fk)).$
It is expected that 
\[AZ(X) = X(q^{\frac{1}{2}{\it h}}, ^t\fts, ^t\mathcal{L}_{Ft}).\]
\end{conjecture}
\vskip .1in

\begin{remark}
 The conjecture is known to hold for all  representations with Iwahori-fixed vectors by \cite{em}.
\end{remark}
\vskip .2in

\subsection{Arthur Parameters}\label{ap}\

\begin{definition} \cite[Section 6]{Art89} An Arthur parameter for $\G$ is a continuous homomorphism
\[\psi: W'_{\fk}\times\text{SL}(2) \lra \G^{\vee},\] such that
\begin{enumerate}
\item the restriction $\psi\vert_{W'_{\fk}}$ is a tempered Langlands parameter, and
\item the restriction $\psi\vert_{\text{SL}(2)}$ is algebraic.
\end{enumerate}

\end{definition}

Attached to an Arthur parameter $\psi$,  is a Langlands parameter $\phi_{\psi}$
and an infinitesimal character $\G^{\vee}\cdot\lambda_{\psi}$ given by
\begin{align*}
\phi_{\psi}: W'_{\fk}&\lra \G^{\vee}\\
w & \mapsto \psi \big(w, \begin{pmatrix} ||w||^{\frac{1}{2}} & 0\\  0 & ||w||^{-\frac{1}{2}}\end{pmatrix}\big);\\
\lambda_{\psi}: W_{\fk} &\lra \G^{\vee},    \;
\lambda_{\psi} := \phi_{\psi}\vert_{W_{\fk}}.
\end{align*}
\vskip .1in

It is expected that for each Arthur parameter $\psi$  there are  sets, consisting of irreducible irreducible representations
of $G(\fk)$,  $\Pi_{\psi}^{\text{Art}}(\G(\fk))$  (Arthur packet) and $\Pi_{\phi_{\psi}}^{\text{Lang}} (\G(\fk))$ (Langlands' packet), satisfying various conditions. In particular, it is conjectured that
\begin{align*}
&\Pi_{\phi_{\psi}}^{\text{Lang}} (\G(\fk))\subset\Pi_{\psi}^{\text{Art}}(\G(\fk));\\
&\Pi_{\psi}^{\text{Art}}(\G(\fk)) \text{ consists of irreducible unitary representations.}\\
\end{align*}

There is no general definition of $\Pi_{\psi}^{\text{Art}}(\G(\fk)).$
Arthur parameters with $\psi\vert_{W'_{\fk}}$  trivial are called {\it basic Arthur parameters.}
The set of such parameters is in bijection with the set of nilpotent orbits on $\fgv.$ For  a nilpotent orbit $\mathcal O \subset \fgv,$
in \cite{cmbo}, the authors defined  the notion of  {\it  basic Arthur packet}.
They  also introduced a larger set $\Pi^{\text{weak}}_{\psi_{\mathcal O}}(\G(\fk)),$ 
 {\it a  weak Arthur packet}. They  conjectured that {\it weak Arthur packets} are unions of {\it Arthur packets.}
 This is Conjecture \ref{cj2}, which uses the language of {\it simplified Arthur parameters.}

\begin{definition} A simplified Arthur parameter is a continuous homomorphism

\begin{align*}
&\widetilde{\psi}: W_{\fk}\times \text{SL}_{\text{Lang}}(2)\times\text{SL}_{\text{Art}}(2) \lra \G^{\vee}\\
&\widetilde{\psi}(W_{\fk}) \text{  is compact and consists of semi-simple elements,}\\
&\widetilde{\psi}\vert_{ \text{SL}_{\text{Lang}}},\text{ and }  \widetilde{\psi}\vert_{\text{SL}_{\text{Art}}}   \text{ are algebraic.}
\end{align*}
\end{definition}
There is a notion of infinitesimal parameter $\Lambda_{\widetilde{\psi}}: W_k \lra \G^{\vee}$  attached to a simplified Arthur parameter $\widetilde{\psi},$
see \cite[Section 4.1]{cfmmx}.

\vskip .2in

\begin{remark}
 The set of Arthur parameters is in bijection with the set of simplified Arthur parameters,
 see \cite[page 278]{Kn}. This bijection at the level of parameters induces a bijection between
 Arthur packets and simplified Arthur packets. The notions of infinitesimal parameter  and infinitesimal character
 do not match. 
 \end{remark}
 
 Relevant to Conjecture \ref{cj2} is the set  of simplified Arthur parameters $\widetilde{\psi}$ such that
 \begin{equation}\label{cn}
 \widetilde{\psi}\vert_{W_{\fk}} \text{ is trivial}, \Lambda_{\widetilde{\psi}}(I_k) ={ \bf 1}, \text{ and }
 \Lambda_{\widetilde{\psi}}(Fr) = q^{\frac{1}{2}\it{h}},
 \end{equation}
 where $\it{h}$ is the middle element of an $\mathfrak{sl}(2)$-triple.
 For such a parameter 
$\widetilde{\psi}$ 
 set
\begin{alignat}{4}\label{simp}
&\widetilde{\psi}_1 &= \widetilde{\psi}\vert_{\text{SL}_{\text{Lang}}\times 1} \qquad\qquad &\widetilde{\psi}_2 &= \widetilde{\psi}\vert_{1 \times \text{SL}_{\text{Art}}}\\\nonumber
&{\it x}_{\widetilde{\psi}} &= d\widetilde{\psi}_1 \begin{pmatrix}  0 & 1\\0 & 0\end{pmatrix}  \qquad \qquad &{\it y}_{\widetilde{\psi}} &= d\widetilde{\psi}_2 \begin{pmatrix}  0 & 1\\0 & 0\end{pmatrix} \\\nonumber 
&{\nu}_{\widetilde{\psi}}  &= d\widetilde{\psi}_1 \begin{pmatrix}  0 & 0\\1 & 0\end{pmatrix}  \qquad\qquad   &\xi_{\widetilde{\psi}} &= d\widetilde{\psi}_2 \begin{pmatrix} 
0 & 0\\1 & 0\end{pmatrix}.
\end{alignat}

\begin{theorem}\cite[Proposition 6.6.1, Lemma 6.4.2]{cfmmx}\label{cun}
If  $\widetilde{\psi}$ is a simplified Arthur parameter satisfying (\ref{cn}), then
\begin{enumerate}
\item $\S_{\psi} := \G^{\vee}({\it h}) \cdot {\it x}_{\widetilde{\psi}} \subset \fgv(2)$
\item $ \G^{\vee}({\it h}) \cdot ({\it x}_{\widetilde{\psi}}, \xi_{\widetilde{\psi}}) $ is open and dense  in $T^*_{\S_{\psi} }(\fgv(2)),$ 
\item $ \G^{\vee}({\it h}) \cdot {\xi}_{\widetilde{\psi}} = S_{\psi}^{Pt}.$
\end{enumerate}
\end{theorem}
 \vskip .1in
 
  We identify the simplified Arthur parameter $\widetilde{\psi}$ with the pair of orbits
 \[(\S_{\psi}, ^t\S_{\psi}^{Pt}) = (\S_{\psi}, \widehat{\S_{\psi}}).\]

 If $\widetilde{\psi}$ is a simplified Arthur parameter, then so is 
 \[\widetilde{\psi}^{Tr}:  (w, {\it x}, {\it y}) \mapsto \widetilde{\psi} (w, {\it y}, {\it x}) \text{ for }  {\it x}, {\it y} \in \text{SL}_2. \]
It follows that $\widetilde{\psi}^{Tr}$
  is parameterized by
$(\widehat{\S_{\psi}}, \S_{\psi}),$ see \cite[Corollary 6.6.2]{cfmmx}. 
 
 \vskip .2in
 
  \begin{remark}
 \begin{enumerate}
 \item  Theorem \ref{cun} is analogous to \cite[Proposition 22.9]{abv} for real groups.
 \item In the context of prehomogeneous vectors spaces, the property of a group having a dense orbit in
 a conormal bundle has been studied from the point of view of $\mathcal D$-modules, see the notion of good Lagrangians \cite[Definition 4.5]{micro}, (see also \cite{lw}).
 \item It is expected that
 \[AZ\big(\Pi_{\psi}^{\text{Art}} (\G(\fk))\big) = \Pi_{\psi^{Tr}}^{\text{Art}}(\G(\fk)).\]
 \end{enumerate}
  \end{remark}

\vskip .2in

 \subsection{A conjecture by Ciubotaru, Mason-Brown and Okada}\label{cmbo}

The set of $\G^{\vee}$-equivalence classes of {\it basic Arthur parameters}  is in bijection with the set of 
$\G^{\vee}$-nilpotent orbits  on $\fgv.$
For $\mathcal O \subset \fgv$ we write $\psi_{\mathcal O}$ for the corresponding {\it basic Arthur parameter.} Then,
 Definition 2.7.1 in \cite{cmbo} reads
 \[ \Pi_{\psi_{\mathcal O}}^{\text{Art}}(\G(\fk)) = \{ X \in \Pi_{q^{\frac{1}{2}\it{h}}}^{\text{Lus}}(\textmd{G}(\fk)):
 AZ(X) \text{ is tempered}\}.\]
 \vskip .2in

The notion of  p-adic {\it weak Arthur packet} is defined in
analogy with the case of  reductive groups defined over $\R$, \cite[Chapter 23]{abv}.
An explicit description is given in \cite[Corollary 3.1.1]{cmbo}. For a nilpotent orbit $\mathcal O \subset \fgv,$ we write
${\it sp}(\mathcal O)$ for the special piece  of $\mathcal O$, in the sense of  \cite{sp}. Then, 
  
\[\Pi^{\text{weak}}_{\psi_{\mathcal O}}(\G(\fk))  = \{ AZ\big(X(q^{\frac{1}{2}}, \S, \mathcal L) \big) \in 
\Pi_{q^{\frac{1}{2}{\it h}}}^{\text{Lus}}\big(\textmd{G}(\fk)\big) : \G^{\vee}\cdot \S \in {\it sp}(\mathcal O)\}.\]

\vskip .2in

\begin{conjecture}\cite[Conjecture  3.1.2]{cmbo}\label{cj2}
Let $\{{\it e}, {\it h}, {\it f} : {\it h}\in \ftv_r\} $ be a $\mathfrak{sl}_2$-triple  attached to a nilpotent orbit $\mathcal O \subset \fgv.$    There is a set  $\mathcal T$  of simplified Arthur parameters satisfying conditions (\ref{cn})
and an appropriate notion of Arthur packet such that
\[\Pi^{\text{weak}}_{\psi_{\mathcal O}}(\G(\fk))  = \underset{\widetilde{\psi}_j \in \mathcal T}{\bigcup} 
\Pi_{\widetilde{\psi}_j}^{\text{Art}}(\G(\fk)).\]
\end{conjecture}
\vskip .2in

In this paper, we verify Conjecture \ref{cj2}
when $\G^{\vee} = F_4$ and $\mathcal O$ is the $F_4(a_3)$-nilpotent orbit. In Section \ref{last}, we show in general that if  Conjecture \ref{cj1} holds, then so does Conjecture \ref{cj2}.
\vskip .2in

\subsection{Micro-Packets}\label{mp}

By Theorem \ref{dll},  $\Pi_{q^{\frac{1}{2}{\it h}}}^{\text{Lus}}(\textmd{G}(\fk))$
is in bijection with the set
\begin{equation*}
\{IC(\S, \mathcal L) : \S \subset \fgv(2) \text{ and } \mathcal L \text{ is a simple equivariant local system on the orbit } \S\}.
\end{equation*}
\vskip .1in
To each such intersection complex $IC(\S, \mathcal L)$, we can attach
non-negative integers $\chi^{\text{mic}}_{\S'}(IC(\S, \mathcal L))$ for any orbit $\S'\subset \fgv(2)$ such that the {\it characteristic cycle}
of  $IC(\S, \mathcal L)$ is 
\begin{equation}\label{cc1}
 CC(IC(\S, \mathcal L)) = \underset{\S'}{\sum} \;\chi^{\text{mic}}_{\S'}(IC(\S, \mathcal L))\; \overline{T^*_{\S'}(\fgv(2))}.
\end{equation}

If $\S'\subset \fgv(2)$ is a $\G^{\vee}(\it{h})$-orbit,  in analogy to \cite[Definition 22.6]{abv}, we define 
\[\Pi^{\text{mic}}_{\S'} = \{ X(q^{\frac{1}{2}{\it h}}, \S, \mathcal L) \in \Pi_{q^{\frac{1}{2}\it{h}}}^{\text{Lus}}(\textmd{G}(\fk)): \chi^{\text{mic}}_{\S'}(IC(\S, \mathcal L)) \neq 0\}.\]

\vskip .2in

When $\G^{\vee}= F_4$  and $\mathcal O = F_4(a_3),$ we compute the characteristic cycles of all relevant
IC-complexes. The computation leads to explicit descriptions of the relevant micro-packets. We verify that these micro-packets satisfy all
the expectations listed in subsection \ref{ap}.
\vskip .3in

We conclude this section by briefly describing the techniques we use to  compute micro-local multiplicities.

\begin{definition}\label{local}
If ${\bf P} \in \text{Per}_{\G^{\vee}({\it h})}(\fgv(2))$ is a simple perverse sheaf  and $\S \subset \fgv(2)$ is a
$\G^{\vee}({\it h})$-orbit, define
\begin{equation}\label{ml}
\chi^{\text{loc}}_{\S}({\bf P}) = \underset{i}{\sum} (-1)^i \; \text{dim}\big( H^i({\bf P})\vert_{\S}\big).
\end{equation}
\end{definition}

By \cite[Theorem 6.3.1]{ka}, \cite[Theorem 8.2]{gz}, there are integers $c(\S, \S')$ so that
\begin{equation}\label{imp} \chi^{\text{mic}}_{\S}({\bf P}) = \underset{\S'}{\sum} c(\S, \S') \; \chi^{\text{loc}}_{\S'}({\bf P}).
\end{equation}
\vskip .1in

[Note that Kashiwara's index formula exhibits $\chi^{\text{loc}}_{\S}({\bf P})$ as an integer combination of various
 $\chi^{\text{mic}}_{\S'}({\bf P}).$ The resulting matrix is inverse to the matrix $\big(c(\S, \S') \big).$]
\vskip .2in

First,  for ${\bf P} = IC(\S, \mathcal L) \in \text{Per}_{\G^{\vee}({\it h})}(\fgv(2)),$ we use the tables of Kazhdan-Lusztig polynomials in \cite{ci08} to determine the local multiplicities  $\chi^{\text{loc}}_{\S}({\bf P}).$ 
 Next, we compute the matrix  $\big(c(\S, \S')\big).$  The Fourier Transform of Subsection \ref{FT} plays a key role
 in our computations. 
 By \cite[Theorem 3.2]{hk} (see also \cite[Equation (4.15)]{lw}), if 
$FT( IC(\S, \mathcal L)) = IC(\S_{Ft}, \mathcal L_{Ft})$, then
 \begin{equation}\label{ft}
 CC(IC(^t \S_{Ft}, ^t \mathcal{L}_{Ft}) ) \, = \, \underset{\S'}{\sum} \;\chi^{\text{mic}}_{\S'}(IC(\S, \mathcal L))\; \overline{T^*_{\widehat{\S'}}(\fgv(2))}.
\end{equation}
Combining equations (\ref{cc1}), (\ref{ml}) and equation (\ref{ft}) we obtain a linear system of equations on the unknowns
 $c(\S, \S')$ that we solve.
 
 \vskip .2in
 \section{Microlocal Packets for $F_4$-split}
 
 \subsection{}
 
 Throughout  this section  $\G = F_4$  and $\G(\fk)$ stands for  the split form of the simple exceptional  p-adic group of type $F_4.$
 We choose a Cartan subalgebra   $\mathfrak h^{\vee} \subset \fgv$ 
 and  an ordered set of simple roots of $\Delta(\fgv, \mathfrak h^{\vee})$, $\{\alpha_1^{\vee}, \alpha_2^{\vee}, \alpha_3^{\vee}, \alpha_4^{\vee}\},$ so that 
 $\alpha_1^{\vee}, \alpha_2^{\vee}$  are long roots.
 \vskip .2in
 
 We let $\mathcal O \subset \fgv$ denote the  complex $\G^{\vee}$-nilpotent orbit of type $F_4(a_3).$ 
 Attached to $\mathcal O$  is a $\mathfrak{sl}_2$-triple $\{{\it e}, {\it h}, {\it f}\}$ with ${\it h}\in \mathfrak h^{\vee}.$
 The semisimple operator $\text{ad}({\it h})$ induces a Lie algebra grading
 \[ \fgv = \fgv(-6) \oplus \fgv(-4)\oplus \fgv(-2)\oplus \fgv(0)\oplus \fgv(2) \oplus )\fgv(4)\oplus \fgv(6),\]
 where $\text{ dim }\fgv(2) = 12,  \text{ dim } \fgv(4) = 6,$ and $\text{ dim } \fgv(6) = 2.$
 
The connected  Levi subgroup  $\G^{\vee}(\it{h})$ has Lie algebra $\fgv(0)$
 and it acts on $\fgv(2)$  by conjugation. The  orbit $\G^{\vee}(\it{h}) \cdot {\it e}$
 is dense in $\fgv(2).$ 
   The pair $(\G^{\vee}(\it{h}), \fgv(2))$ is a   prehomogeneous vector space equivalent to
  \[\big(\text{GL}(2,\C)\times\text{GL}(3,\C), \C^2\otimes \text{Sym}_2(\C^3)\big).\]
  
  The  unique monic  semi-invariant function  $f$ of 
  $\big(\text{GL}(2,\C)\times\text{GL}(3,\C), \C^2\otimes \text{Herm}_2(\C^3)\big),$ 
  has degree   $12,$  and by \cite{km}, the Bernstein-Sato polynomial (or $b$-function)
of $f$ is 
\begin{equation}\label{eq:bfun}
b_f(s) = (s+1)^4 (s+3/4)^2 (s+5/4)^2 (s+5/6)^2(s+7/6)^2.
\end{equation}

 \vskip .2in
 
 There are  twelve  $\G^{\vee}(\it{h})$-orbits  on $\fgv(2).$ We denote them by  $\S_i$, with $i=0,\dots, 11$. The notation is so that $i\leq j$ implies that $\dim \S_i \leq \dim \S_j.$ The following is the Hasse diagram with respect to containment of orbit closures. See, for example \cite[Section 5]{ci08}:

\[\xymatrix@R-1.4pc{
12: & & \S_{11}^* \ar@{-}[d] & \\
11: & & \S_{10}' \ar@{-}[dl] \ar@{-}[dr] & \\
10: & \S_8 \ar@{-}[dd]\ar@{-}[dr]& & \S'_9 \ar@{-}[dl] \\
9: & & \S_7 \ar@{-}[d]\ar@{-}[ddr] & \\
8: & \S_6 \ar@{-}[d]  & \S_5' \ar@{-}[dl]& \\
7: & \S_3 \ar@{-}[dr] & & \S_4 \ar@{-}[dl] \\
6: & & \S_2' \ar@{-}[d]& \\
4: & & \S_1 \ar@{-}[d]& \\
0: & & \S_0 & \\
}\]
Here, the left column designates the dimensions of orbits. We put a prime symbol $'$ above the orbits that have component groups isomorphic to $\Z/2\Z$, and remark that $\S_{11}$ has component group isomorphic to $S_4$, designated by the star symbol. All the other orbits have trivial component groups.
The open orbit $\S_{11}$ has four equivariant simple local systems indexed by partitions of $4$.
Each partition ${\bf p}$ labels an irreducible representation $\pi_{\bf p}$  of $S_4.$  Their dimensions are
$\text{dim}\; \pi_{(4)} = \text{dim}\; \pi_{(1)^4} = 1; \text{dim}\; \pi_{(31)} = \text{dim}\; \pi_{(211)} = 3$
and $\text{dim}\; \pi_{(22)} = 2.$
\vskip .1in

Recall the notation of  subsection \ref{od} and the involution on the set
of orbits on  $\S\subset \fgv(2)$ given by $\S \mapsto \widehat{\S} = ^t\S^{Pt}.$  
In this example, a direct computation from the definition (or an inspection of the holonomic diagram in \cite{km}) yields 
\begin{equation*}
\begin{tabular}{|c | c| c|  c| c | c | c|}
\hline
$\widehat{\S_{11}} = \S_0 $ & $\widehat{\S_{10}} = \S_1$ &  $\widehat{\S_{9}} = \S_2$ & $\widehat{\S_{8}} = \S_3$&
$\widehat{\S_{7}} = \S_7$ & $\widehat{\S_{6}} = \S_5$ & $ \widehat{\S_{4}} = \S_4$\\
\hline
\end{tabular}
\qquad \text{Table :1}
\end{equation*}
\vskip .2in

\subsection{Unipotent representations of $F_4$-split}\label{danaz}
The results of this subsection are all due to  D.~Ciubotaru, see \cite{ci08}, \cite{ci22} and \cite[Section 4]{cmbo}.
 There are $20$ non-isomorphic irreducible representations of $F_4$-split at infinitesimal
 character $q^{\frac{1}{2}{\it h}},$ where ${\it h}$ is the middle element of the $\mathfrak{sl}_2$-triple
 attached to $\mathcal O$ of type $F_4(a_3).$ We list the representations as  $X_1 \ldots X_{20}.$ 
All the representations in the list  but $X_5$ have Iwahori-fixed vectors. 
There is exactly one non-unitary representation on this list, $X_{16}$, with geometric parameter
$(\S_4, (1)).$

In Table: 2 below, we  summarize information relevant to our work.
\begin{equation*}
\begin{tabular}{|c| c||c|c|}
\hline
$\text{Rep} $ & $\text{Geom. Param.} $& AZ & $\text{Geom. Param.}$\\
\hline
$X_1$ & $(\S_{11}, (4))$ & $X_{20}$& $(\S_0, (1))$\\
\hline
$X_2$ & $(\S_{11}, (31))$ & $X_{19}$& $(\S_1, (1))$\\
\hline
$X_3$ & $(\S_{11}, (22))$ & $X_{17}$& $(\S_2, (1))$\\
\hline
$X_4$ & $(\S_{11}, (211))$ & $X_{13}$& $(\S_5, (1)^2)$\\
\hline
$X_5$ & $(\S_{11}, (1)^4)$ & $X_{5}$& $(\S_{11}, (1)^4)$\\
\hline
$X_6$ & $(\S_{10}, (1))$ & $X_{15}$& $(\S_{3}, (1))$\\
\hline
$X_7$ & $(\S_{10}, (1)^2)$ & $X_{9}$& $(\S_{7}, (1))$\\
\hline
$X_8$ & $(\S_{8}, (1))$ & $X_{8}$& $(\S_{8}, (1))$\\
\hline
$X_{10}$ & $(\S_{9}, (1))$ & $X_{18}$& $(\S_{2}, (1)^2)$\\
\hline
$X_{11}$ & $(\S_{9}, (1)^2)$ & $X_{11}$& $(\S_{9}, (1)^2)$\\
\hline
$X_{12}$ & $(\S_{5}, (1))$ & $X_{14}$& $(\S_{6}, (1))$\\
\hline
$X_{16}$ & $(\S_{4}, (1))$ & $X_{16}$& $(\S_{4}, (1))$\\
\hline
\end{tabular}
\quad \text{Table: 2}
\end{equation*}
\begin{remark}\label{relevant}
 Conjecture \ref{cj1} holds for Iwahori-Spherical representations, see \cite{em}.
 It follows that if $(\S, \S')$ occurs in a row of Table:2,  then
  $\S' = ^t\S_{Ft}.$
 \end{remark}
 
 \subsection{Characteristic Cycles}
 
  In this subsection we set $V =\fgv(2).$  We compute the characteristic cycles of the
 IC-complexes, $\text{IC}(\S, \mathcal L)$ for  $(\S, \mathcal L)$ in Table: 2. 
 \begin{theorem}\label{main}
 
 Let $c := C(\S_4, \S_{11}).$ Then,  $c \geq 2$ and we have:
 \begin{align*}
 \CC(IC(\S_0, 1))  & =  [\ol{T_{\S_0}^* V}],\\
    \CC (IC(\S_1, 1)) & =  [\ol{T_{\S_1}^* V}]+ 3 \; [\ol{T_{\S_0}^* V}],\\
    \CC (IC(\S_2, 1)) & =  [\ol{T_{\S_2}^* V}]+ [\ol{T_{\S_1}^* V}]+ 2\; [\ol{T_{\S_0}^* V}],\\
    \CC (IC(\S_2, 1^2)) & =  [\ol{T_{\S_2}^* V}],\\
    \CC (IC(\S_3, 1)) & =  [\ol{T_{\S_3}^* V}]+ [\ol{T_{\S_1}^* V}],\\
    \CC(IC(\S_4,1) & =  [\ol{T_{\S_4}^* V}],\\
    \CC (IC(\S_5, 1)) & =  [\ol{T_{\S_5}^* V}],\\
    \CC (IC(\S_5, 1^2))& =  [\ol{T_{\S_5}^* V}]+ [\ol{T_{\S_3}^* V}]+ [\ol{T_{\S_2}^* V}]+ 2\; [\ol{T_{\S_1}^* V}]+
    3\; [\ol{T_{\S_0}^* V}], \\
    \CC (IC(\S_6, 1)) & =  [\ol{T_{\S_6}^* V}],\\
    \CC (IC(\S_7, 1))& =  [\ol{T_{\S_{7}}^* V}]+ (c+1)\; [\ol{T_{\S_{4}}^* V}]+ [\ol{T_{\S_{3}}^* V}]+ 2\; [\ol{T_{\S_{2}}^* V}]+[\ol{T_{\S_{1}}^* V}],\\
    \CC (IC(\S_8, 1))& = [\ol{T_{\S_8}^* V}]+ [\ol{T_{\S_7}^* V}]+ 2\; [\ol{T_{\S_6}^* V}]+ 2\; [\ol{T_{\S_5}^* V}]+\; (c-2)[\ol{T_{\S_4}^* V}]+ [\ol{T_{\S_3}^* V}],\\
    \CC (IC(\S_9, 1))& = [\ol{T_{\S_9}^* V}],\\
    \CC (IC(\S_9, 1^2)) & =  [\ol{T_{\S_9}^* V}]+ [\ol{T_{\S_7}^*V}]+ c\; [\ol{T_{\S_4}^* V}]+ [\ol{T_{\S_2}^* V}],\\
    \CC (IC(\S_{10}, 1)) & =  [\ol{T_{\S_{10}}^* V}]+ [\ol{T_{\S_8}^* V}],\\
    \CC (IC(\S_{10}, 1^2)) & = [\ol{T_{\S_{10}}^* V}]+ 2\; [\ol{T_{\S_{9}}^* V}]+ [\ol{T_{\S_{8}}^* V}]+ [\ol{T_{\S_{7}}^* V}]+ (c +1) \; [\ol{T_{\S_{4}}^* V}],\\
    \CC (IC(\S_{11}, (4) )) & = [\ol{T_{\S_{11}}^* V}],\\
  \CC (IC(\S_{11}, (31)) & = 3\; [\ol{T_{\S_{11}}^* V}]+ [\ol{T_{\S_{10}}^* V}],\\
     \CC (IC(\S_{11}, (22)) & = 2\; [\ol{T_{\S_{11}}^* V}]+ [\ol{T_{\S_{10}}^* V}]+ [\ol{T_{\S_{9}}^* V}],\\
    \CC (IC(\S_{11}(211)) & = 3\; [\ol{T_{\S_{11}}^* V}]+ 2\; [\ol{T_{\S_{10}}^* V}]+ [\ol{T_{\S_{9}}^* V}]+ [\ol{T_{\S_{8}}^* X}]+[\ol{T_{\S_{6}}^* V}].
    \end{align*}
\end{theorem}
 
 \begin{proof}
 In order to compute
 \begin{equation*}
 CC(IC(\S, \mathcal L)) = \underset{\S'}{\sum} \chi^{\text{mic}}_{\S'}(IC(\S, \mathcal L))\; \overline{T^*_{\S'}(V)},
\end{equation*}
we need to determine each micro-local multiplicity $\chi^{\text{mic}}_{\S'}(IC(\S, \mathcal L)).$
We write, as in  sub-section \ref{mp},
 \begin{equation}\label{eulo}
  \chi^{\text{mic}}_{\S'}({\bf P}) = \underset{\S''}{\sum} c(\S', \S'') \; \chi^{\text{loc}}_{\S''}({\bf P}).
  \end{equation}
 Each local system $\mathcal L$ on $\S$ determines an irreducible representation of the component group of $\S.$
 Recall that for all $\S \neq \S_{11}$ such representations are one-dimensional. When $\S = \S_{11},$ $\text{dim}\; \pi_{(4)} = \text{dim}\; \pi_{(1)^4} = 1; \text{dim}\; \pi_{(31)} = \text{dim}\; \pi_{(211)} = 3$
and $\text{dim}\; \pi_{(22)} = 2.$
\vskip .1in

As a first step, we compute all local multiplicities.
For two local systems  $(\S, \mathcal L), (\S'', \mathcal L'')$  with $\S, \S'' \subset \fgv(2),$ we write $d(\S) = \text{dim}(\S)$
and
\begin{equation}\label{kl}
\mathcal P_{(\S'', \mathcal L''), (\S, \mathcal L)}(q) = \underset{i}{\sum} \; [\mathcal L'': \mathcal H^{-d(\S)+2i}(IC(\S, \mathcal L))\vert_{\S''}]  \; q^i,
\end{equation}
for the corresponding Kazhdan-Lusztig polynomial. Then, we have
\begin{equation}\label{kluse}
\chi^{\text{loc}}_{\S''}(IC(\S, \mathcal L)) = (-1)^{\text{dim}(\S)}
\; \underset{\mathcal L''}{\sum} \; \text{dim}(\mathcal L'')\; \mathcal P_{(\S'', \mathcal L''), (\S, \mathcal L)}(1).
\end{equation}

We use Dan Ciubotaru's   Kazhdan-Lusztig polynomials tables
in \cite[Section 5.1]{ci08} to complete the first stage of our computations.

\vskip .1in

Next,  we  recall that

\begin{equation}\label{ident}
\chi^{mic}_{\S'}(IC(\S, \mathcal L)) = \chi^{mic}_{\S'^{Pt}}(IC(\S_{Ft}, \mathcal L_{Ft})) = \chi^{mic}_{\widehat{\S'}}(IC(^t\S_{Ft}, ^t\mathcal L_{Ft})),
\end{equation}
where $\widehat{\S'} =  \;^t(\S')^{Pt}.$
For $\S \subset \fgv(2),$ $^t\S_{Ft}$ is given in 
 Table:2 while Table:1 gives   $\widehat{\S'}.$ 
 \vskip .1in
 
Combining equations  (\ref{eulo}),  (\ref{kluse}) and (\ref {ident})  we obtain a linear system of equations on the unknowns $c(\S', \S'')$ that we solve.
  The detailed computations are included in  Appendix A.
\end{proof}
 
 \vskip .2in
 
\begin{remark} In   Appendix A , Lemma \ref{l4},  we show that $c = c(\S_4, \S_{11}).$ 
 We are unable to determine the value of $c.$  
 It is noteworthy  that $T^*_{\S_4} V$ is the only conormal bundle in our list that does not contain
 a dense $\G^{\vee}(\it{h})$-orbit.
  \end{remark}
 \vskip .2in
 
 Theorem \ref{main} lists the  characteristic cycles of the Iwahori-Spherical representations attached to 
 the minimal distinguished nilpotent orbit of type $F_4.$ The proof of the Theorem relies on the Kazhdan-Lusztig tables
 in \cite{ci08}.
  The pair $(\S_{11},(1^4))$ is the geometric parameter of representation $X_5,$ which is not Iwahori-Spherical.
  The tables in \cite{ci08}, do not include the relevant Kazhdan-Lusztig needed to compute the $\CC(IC(\S_{11},(1^4)).$
  We use $D$-module techniques to settle the following proposition.

 \begin{prop} \label{missing}
  \[\CC (IC(\S_{11},(1^4)))  = \sum_{\substack{i=0 \\  i \neq 4}}^{11}\; [\ol{T_{\S_i}^* V}] \,\, +  c(\S_4, \S_{11}) \;  [\ol{T_{\S_4}^* V}] \] 
  
 \end{prop}
 
\begin{proof}
 Set $V = \fgv(2)$ and let $\mathcal M=(\mathcal O_V)_f$ denote the localization of the structure sheaf $\mathcal O_V$ at the semi-invariant $f$. First, we show that  $\CC (IC(\S_{11},(1^4))) = \CC (\mathcal{M})$. 
 \vskip .1in

 The complement $U=V\setminus \{f=0\}$ is  $\S_{11}$, the dense  $\G^{\vee}(\it{h})   $-orbit  on  $V.$ Under the Riemann--Hilbert correspondence, the perverse sheaf $IC(\S_{11},(1^4))$ corresponds to the minimal extension to $V$ of the simple (regular, holonomic) $\mathcal D_U$-module $\mathcal{O}_U \cdot \sqrt{f}$.  This minimal extension is known to be the
 unique simple submodule of the $\mathcal{D}_V$-module  $(\mathcal O_V)_f \cdot \sqrt{f}$. 
 It is important to observe  from  (\ref{eq:bfun}) that  the $b$-function of $f$ does not have any half-integer roots.
 Then, the assumptions of  \cite[Corollary 3.13]{ly}  are satisfied and we can conclude that   $(\mathcal O_V)_f \cdot \sqrt{f}$ is simple as a $\mathcal D_V$-module. It follows that the Riemann-Hilbert correspondence attaches  
 $IC(\S_{11},(1^4))$ to $(\mathcal O_V)_f \cdot \sqrt{f}.$  In particular, we have
 \[\CC (IC(\S_{11},(1^4)) = \CC ((\mathcal O_V)_f \cdot \sqrt{f}).\]
 On the other hand, by  \cite[Theorem 3.2]{gz} (see also \cite[Lemma 1.11]{lo}) 
 \[ \CC ((\mathcal O_V)_f \cdot \sqrt{f}) = \CC (\mathcal M).\]

 \vskip .1in

  Next, we determine $\CC(\mathcal M)$. By  \cite{km}, each $\ol{T_{\S_i}^* V}$  with $1\leq i \leq 11$ is a component of the characteristic variety of $\S_{11}$, and $\ol{T_{\S_j}^* V}$ has a dense $\G^{\vee}(\it{h})$-orbit if and only if $j\neq 4$ (see also Theorem \ref{cun}).  This, together with \cite[Proposition 4.7]{micro} (or \cite[Lemma 3.12 and proof of Proposition 3.14]{lw}) implies that  if $j\neq 4$,  then $\ol{T_{\S_j}^* V}$ has multiplicity one in $\CC(\mathcal M)$.  Thus, we have
\begin{equation}\label{ccloc}
CC (\mathcal{M} )=  \sum_{\substack{i=0 \\  i \neq 4}}^{11} \; [\ol{T_{\S_i}^* V}] \,\, + a\;  [\ol{T_{\S_4}^* V}],
\end{equation}
for some integer $a\geq 1$. 
\vskip .1in

We are left to show that $a=c$. Note that $\mathcal M$  is the direct image (in the category of  equivariant $D$-modules) of the trivial local system on $\S_{11}.$  All its composition factors must correspond to one of the equivariant irreducible local systems from Table: 2. 
Let  $\bf {P}_{\mathcal M}$ be the perverse sheaf associated to $\mathcal M$ via the
Riemann--Hilbert correspondence. 

Using Theorem \ref{main} and (\ref{ccloc}), we see that 
 
\begin{align}\label{cs}
&[\mathbf{P}_{\mathcal M}] = [IC(\S_{11}, (4))] + [IC(\S_{10}, 1)]+ [IC(\S_{9}, 1^2)]+ [IC(\S_{6}, 1)]+[IC(\S_{5}, 1)]+\\ \nonumber
& + (a-c) [IC(\S_{4}, 1)]+ [IC(\S_{3}, 1)]+[IC(\S_{0}, 1)].
\end{align}
We observe that  $a=c$ if and only if the multiplicity of the irreducible perverse sheaf
 $[IC(\S_{4}, 1)]$ in the composition series of $[\bf{P}_{\mathcal M}]$ is zero. 
We use \cite[Corollary 1.25(a) and Corollary 15.13(a)]{abv} to show that this is the case.
\vskip .1in

 Identifying the Grothendieck group of constructible and perverse sheaves, and  writing $\gamma = (\S_{\gamma}, \mathcal L_{\gamma})$, we have
\begin{align*}
[\mu(\S_{11},(4)) ] &= \underset{\text{local systems }\gamma}{\sum} m_g(\gamma, (\S_{11},(4)))  [IC(\gamma)]\\ 
&= [IC(\S_{11}, (4))] + [IC(\S_{10}, 1)]+ [IC(\S_{9}, 1^2)]+ [IC(\S_{6}, 1)]+[IC(\S_{5}, 1)]+\\ \nonumber
&+ (a-c) [IC(\S_{4}, 1)]+ [IC(\S_{3}, 1)]+[IC(\S_{0}, 1)].
\end{align*}

In  the language of \cite{abv},   $(m_g(\gamma, \xi))$  is  the {\it geometric multiplicity matrix.} The inverse of this matrix
$(c_g(\eta, \delta))$ allows us to write perverse sheaves in terms of constructible sheaves. That is,

\[[\mathbf {P}(\gamma)] =   \underset{\text{local systems }\delta}{\sum} (-1)^{\text{dim}(\S_{\delta})}\;
 c_g(\delta, \gamma)  [\mu(\delta)]. \]
 
It follows, see \cite[Corollary 1.25(a) and Corollary 15.13(a)]{abv}, that
\begin{equation*}
 0 =  \underset{\text{local systems }\gamma}{\sum} c_g((\S_4, 1), \gamma)\;m_g(\gamma, (\S_{11},(4)));
\end{equation*}
where the various $c_g((\S_4, 1), \gamma)$ can be computed using
  the Kazhdan-Lusztig tables in \cite{ci08}.
  We have,
  
\begin{align*}
 0 &=  \underset{\text{local systems }\gamma}{\sum} c_g((\S_4, 1), \gamma)\;m_g(\gamma, (\S_{11},(4)) )\\
& = m_g(\S_4, 1), \S_{11},(4)) - 1 + 2 - 1 = (a - c). 
\end{align*}
Thus, $a = c.$

\end{proof}
\vskip .3in

\begin{remark}
Theorem \ref{main} can be settled using other arguments.
Indeed, our first partial computations  relied on the results and methods in \cite{km}, \cite{lw}, \cite{lo}, \cite{ly}, \cite{ll}.
Let us briefly mention some of these ideas. 
For any partition  $\lambda$, the characteristic cycle of the localization $\mathcal M$ along $f$ of the $\D$-module corresponding to $IC(\S_{11}, \lambda)$ can be determined readily (up to $c$). As a mixed Hodge module, 
$\mathcal M$ also carries a weight filtration. Based on the calculations of Bernstein--Sato polynomials in \cite{km} and \cite[Section 4.2]{lo}, and using \cite[Proposition 3.14]{ly}, the weight length of $\mathcal M$ is $4$, whenever $\lambda \neq (1^4)$.  By \cite{ll}, we have an isomorphism $^t Ft(\mathcal M) \cong \mathbb{D}(\mathcal M)$, of mixed Hodge modules (with an appropriate weight shift), where $\mathbb{D}$ stands for the holonomic dual. Using methods in  \cite{lw}, the objects $\mathcal M$ are injective-projective in the category of equivariant $\D$-modules. With some work, one can then obtain the composition series for each $\mathcal M$, and deduce information on  the characteristic cycles of the components. In particular, we record the following fact, of independent interest.
\vskip .1in
 
The non-zero associated graded terms of the weight filtration on $\M= (\mathcal{O}_V)_f$ are
 \[\gr_{12}^W \M = IC(\S_{11}, 1), \, \gr_{13}^W \M = IC(\S_{10}, 1), \, \gr_{14}^W \M = IC(\S_{9}, 1^2) \oplus IC(\S_{6}, 1) \oplus IC(\S_{5}, 1), \]
 \[\gr_{15}^W \M = IC(\S_{3}, 1), \,\, \gr_{16}^W \M = IC(\S_{0}, 1).  \]
 
\end{remark}

\vskip .4in
 
 \subsection{ Weak Arthur Packets $\Pi^{\text{weak}}_{\psi_{F_4(a_3)}}(F_4\text{-split}) $}
 
 The aim of this section is to verify Conjecture \ref{cj2} for
 the {\it weak Arthur packet} 
 \begin{equation*}
 \Pi^{\text{weak}}_{\psi_{F_4(a_3)}}(F_4\text{-split}) =
 \{ AZ\big(X(q^{\frac{1}{2}}, \S, \mathcal L) \big) \in 
\Pi_{q^{\frac{1}{2}{\it h}}}^{\text{Lus}}\big(\textmd{G}(\fk)\big) : \G^{\vee}\cdot \S \subset {\it sp}\big( F_4(a_3) \big)\}.
\end{equation*}

 In the notation of \cite{cmcg}, 
 \begin{align*}
 &\text{Special piece of }(F_4(a_3)) =\{F_4(a_3), C_3(a_1), B_2, A_1+\tilde{A_2}, \tilde{A_1}+A_2\};\\
 &\{\S \subset \fgv(2)  : \G^{\vee}\cdot \S \in {\it sp}(F_4(a_3))\} =
 \{\S_{11}, \S_{10}, \S_9, \S_8, \S_7\};\\
 &\{\widehat{S}: \S \in {\it sp}(F_4(a_3))\}  = \{ \S_0, \S_1, \S_2, \S_3, \S_7\}.
 \end{align*}
 
 It follows from Subsection \ref{danaz} that
  
  \begin{align*}
 \Pi^{\text{weak}}_{\psi_{F_4(a_3)}}(F_4\text{-split})& =\{ X_5 = AZ (X_5) , X_7 = AZ (X_9) , 
 X_8 = AZ(X_8) , X_9 = AZ(X_7), \\
 &X_{11} = AZ(X_{11}), 
 X_{13} = AZ(X_4), X_{15}  = AZ(X_6), X_{17} = AZ(X_3), \\
 &X_{18} = AZ(X_{10}) , X_{19} = AZ(X_2), X_{20} =AZ(X_1)\}.
 \end{align*}

 Conjecture \ref{cj2} states that there is a set of simplified Arthur parameters $\{\widetilde{\psi}_j\}$  and an appropriate notion of
 Arthur packet  so that  $ \Pi^{\text{weak}}_{\psi_{F_4(a_3)}}(F_4\text{-split}) = \underset{\widetilde{\psi}_j }{\bigcup} \;
\Pi^{\text{Art}}_{\widetilde{\psi}_j}.$

 \vskip .2in
 
 There are $10$ simplified Arthur parameters, $\widetilde{\psi}:\text{SL}(2)_{\text{Lang}}\times\text{SL}(2)_{\text{Art}}\mapsto \G^{\vee},$  with  infinitesimal parameter determined by  $q^{\frac{1}{2}{\it h}}.$ We label these parameters by pairs of orbits, as
 indicated in  subsection \ref{ap}. These are:
 \begin{align*}
 \widetilde{\psi}_0 \equiv (\S_0, \S_{11}), \widetilde{\psi}_1 \equiv (\S_1, \S_{10}), \widetilde{\psi}_2 \equiv (\S_2, \S_{9}),
 \widetilde{\psi}_3 \equiv (\S_3, \S_{8}), \widetilde{\psi}_7 \equiv (\S_7, \S_{7});\\
 \widetilde{\psi}^t_1 \equiv (\S_{11}, \S_{0}),  \widetilde{\psi}^t_2 \equiv (\S_{10}, \S_{1}),  \widetilde{\psi}^t_2 \equiv (\S_{9}, \S_{2}), \widetilde{\psi}^t_3 \equiv (\S_{8}, \S_{3}), \widetilde{\psi}^t_7 \equiv (\S_{7}, \S_{7}).
 \end{align*}
 \vskip .1in

  We use
  the  lists   of  geometric Langlands' parameters of the representations $\{X_i\}_{i=1}^{20}$ in Table:2 and
  the lists of  characteristic cycles of the corresponding perverse sheaves in Theorem \ref{main}
to  find:
 \begin{align*}
 \Pi^{mic}_{\S_0} & = \{X_5, X_{13}, X_{17},  X_{19}, X_{20}\},\\
 \Pi^{mic}_{\S_1} & = \{X_5, X_{9}, X_{13}, X_{15},  X_{17}, X_{19}\},\\
  \Pi^{mic}_{\S_2} & = \{X_5, X_{9}, X_{11}, X_{13},  X_{17}, X_{18} \},\\
  \Pi^{mic}_{\S_3} & = \{X_5, X_8, X_9, X_{13},  X_{15}\},\\
 \Pi^{mic}_{\S_7} & = \{X_5,  X_7, X_8, X_9, X_{11}\}.
 \end{align*}
 
 The {\it Basic Arthur Packet} at $ \Pi^{Art}_{\psi_{F_4(a_3)}}(\G(\fk)),$ computed in
 \cite[Section 4]{cmbo}, is 
 \[\Pi^{Art}_{\psi_{F_4(a_3)}}(\G(\fk))  = \{X_5, X_{13}, X_{17},  X_{19}, X_{20}\}.\]
 Thus,
 
 \[\Pi^{Art}_{\psi_{F_4(a_3)}}(\G(\fk)) = \Pi^{mic}_{\S_0} (\G(\fk)).\]

 Moreover, 
 \begin{align*}
 \Pi^{\text{weak}}_{\psi_{F_4(a_3)}}(\G(\fk)) &= \underset{ \{\widehat{S}: \;\; \G^{\vee}\cdot \S \in\;  {\it sp}(F_4(a_3))\}}
 {\bigcup} \Pi^{mic}_{\widehat{\S}} (\G(\fk)) \\
  &=  \Pi^{mic}_{\S_0}  \cup  \Pi^{mic}_{\S_1} \cup  
 \Pi^{mic}_{\S_2} \cup   \Pi^{mic}_{\S_3}  \cup \Pi^{mic}_{\S_7}.
 \end{align*} 
 
 \vskip .1in
 
  By \cite{ci22}, these micro-packets consist of unitary representations.

 \vskip .2in

By Table:2 and Theorem \ref{main}, we have:
 \begin{align*}
AZ\big( \Pi^{mic}_{\S_0}(\G(\fk))\big) & =\{X_1, X_2, X_3, X_4, X_5\} = \Pi^{mic}_{\S_{11}}(\G(\fk))\\
 AZ\big(\Pi^{mic}_{\S_1}(\G(\fk))\big) & =\{X_2, X_3, X_4, X_5, X_6, X_7\} = \Pi^{mic}_{\S_{10}}(\G(\fk))\\
AZ\big(\Pi^{mic}_{\S_2}(\G(\fk))\big) & =\{X_3, X_4, X_5, X_7, X_{10}, X_{11} \} = \Pi^{mic}_{\S_9}(\G(\fk))\\
AZ\big(\Pi^{mic}_{\S_3}(\G(\fk))\big) & =\{X_4, X_5, X_6, X_7, X_{8} \} = \Pi^{mic}_{\S_8}(\G(\fk))\\
 AZ\big(\Pi^{mic}_{\S_7}(\G(\fk))\big) & =\{X_5, X_7, X_8, X_9, X_{11} \} = \Pi^{mic}_{\S_7}(\G(\fk)).
 \end{align*}
 \vskip .1in
 
 We conclude  that $\Pi^{\text{weak}}_{\psi_{F_4(a_3)}}(F_4\text{-split})$ is a union of micro-packets
 that verify the expectations listed in subsection \ref{cmbo}. 
 
 \vskip .3in
 
 \section{Weak Arthur packets}\label{last}
 
 The goal of this section is to show, in general,  that if Conjecture \ref{cj1} holds then
 $\Pi^{\text{weak}}_{\psi_{\mathcal O}}(\G(\fk))$ is a union of micro-packets of simplified Arthur type.
 We keep the notation of Subsection \ref{od}. 
 
 Fix  a nilpotent orbit $\mathcal O \subset \fgv$ and choose a corresponding $\mathfrak{sl}_2$-triple
 $\{{\it e}, {\it h}, {\it f}\}$ with ${\it h} \in \ftv_r.$ Observe that the orbit
 $\G^{\vee}({\it h}) \cdot {\it e}$ is open and dense in $\fgv(2).$
 Write ${\it sp}(\mathcal O)$ 
 for the special piece of $\mathcal O,$ in the sense of \cite{sp}.

 Recall,
 \[\Pi^{\text{weak}}_{\psi_{\mathcal O}}(\G(\fk))  = \{ AZ\big(X(q^{\frac{1}{2}}, \S, \mathcal L)\big) \in 
\Pi_{q^{\frac{1}{2}{\it h}}}^{\text{Lus}}\big(\textmd{G}(\fk) \big) : \G^{\vee}\cdot \S \in {\it sp} (\mathcal O)\}.\]

\vskip .1in

\begin{theorem}
Assume Conjecture \ref{cj1} holds for $\G(\fk).$ Then,
\begin{equation}\label{general}
\Pi^{\text{weak}}_{\psi_{\mathcal O}}(\G(\fk))  = \underset{\substack\{{\widehat{\S}\subset \;\fgv(2):\\
 \;G^{\vee}\cdot \S \in  \;{ \it sp}(\mathcal O)}\}}{\bigcup} \Pi^{\text{mic}}_{\widehat{\S} } \big(\G(\fk)\big).
 \end{equation}
\end{theorem}

\begin{proof}

First, we assume that  $Y \in \Pi_{q^{\frac{1}{2}{\it h}}}^{\text{Lus}}\big(\textmd{G}(\fk)\big)$  belongs to a micro-packet
$ \Pi^{\text{mic}}_{\widehat{\S}} \big(\G(\fk)\big)$ where $\S = \G^{\vee}({\it h})\cdot x$ and 
 $\G^{\vee} \cdot x \in {\it sp}(\mathcal O).$ We argue that $Y \in \Pi^{\text{weak}}_{\psi_{\mathcal O}}(\G(\fk)).$

Since the AZ duality is an involution on   $\Pi_{q^{\frac{1}{2}{\it h}}}^{\text{Lus}}\big(\textmd{G}(\fk)\big)$, we can write 
\[Y = AZ\big( X( q^{\frac{1}{2}}, \S', \mathcal L')\big),\]   for some orbit 
$\S' =  \G^{\vee}({\it h})\cdot {\it x'}  \subset \fgv(2).$
 If Conjecture \ref{cj1} holds true, the geometric Langlands' parameter of $Y$ is the pair
$(^t\S'_{Ft}, ^t\mathcal L'_{Ft}).$ Thus, our assumptions on $Y$ guarantee that
\[\chi^{\text{mic}}_{\widehat{\S}} (IC(^t\S'_{Ft}, ^t\mathcal L'_{Ft})) \neq 0.\]
On the other hand,  identity (\ref{ft}) gives
\[ \chi^{\text{mic}}_{\S} (IC(\S', \mathcal L' ))  = \chi^{\text{mic}}_{\widehat{\S}} (IC(^t\S'_{Ft}, ^t\mathcal L'_{Ft}))   \neq 0.\]

Hence, 
\[\S =  \G^{\vee}({\it h})\cdot x\subset \overline{\S'}= \overline{\G^{\vee}({\it h})\cdot {\it x'}} \subset \fgv(2)
= \overline{ \G^{\vee}({\it h})\cdot {\it e}}.\]

 Since $\G^{\vee}\cdot {\it x} \in {\it sp} (\mathcal O)$ and $\G^{\vee}\cdot {\it x} \subset \overline{\G^{\vee}\cdot {\it x'}},$ 
 we have $\G^{\vee} \cdot {\it x'} \in {\it sp}(\mathcal O).$
 We conclude that $Y \in  \Pi^{\text{weak}}_{\psi_{\mathcal O}}(\G(\fk)).$
 \vskip .1in

If $Y \in \Pi^{\text{weak}}_{\psi_{\mathcal O}}(\G(\fk)),$ 
the explicit description of  $\Pi^{\text{weak}}_{\psi_{\mathcal O}}(\G(\fk))$ in terms of $AZ$ implies 
 that  $Y = AZ\big(X(q^{\frac{1}{2}}, \S, \mathcal L)\big)$ for some orbit $\S$ with $\G^{\vee}\cdot \S \in {\it sp}\big(\mathcal O\big).$ Moreover, if Conjecture \ref{cj1} is true, 
the geometric Langlands' parameter of $Y$ is the pair $(^t\S_{Ft}, ^t\mathcal L_{Ft}).$
By identity (\ref{ft}), 
\[\chi^{\text{mic}}_{\widehat{\S}} (IC(^t\S_{Ft}, ^t\mathcal L_{Ft})) =   \chi^{\text{mic}}_{\S} (IC(\S, \mathcal L)) \neq 0.\]
That is, $Y \in \Pi^{\text{mic}}_{\widehat{\S} } \big(\G(\fk)\big).$
\end{proof}

\vskip .1in

Recall that the set of simplified Arthur parameters  satisfying (\ref{cn}) is 
 parameterized  by  pairs of orbits
 $\big(\S_{\widetilde{\psi}}, \widehat{\S_{\widetilde{\psi}} }\big).$
Write $\mathcal T$ for the set  of simplified Arthur parameters  satisfying (\ref{cn}) such that  $\G^{\vee} \cdot \S_{\widetilde{\psi}}  \in {\it sp}(\mathcal O).$

\begin{corollary} If Conjecture \ref{cj1} holds, we have

\[\Pi^{\text{weak}}_{\psi_{\mathcal O}}(\G(\fk))  = \underset{\widetilde{\psi} \in \; \mathcal T}{\bigcup} 
\Pi^{\text{mic }}_{ \widehat{\S_{\widetilde{\psi}} }}\; \big(\G(\fk)\big).\]
\end{corollary}
\vskip .2in

 \appendix
 \section{ Proof of Theorem \ref{main} }
 
 \begin{lemma}\label{l1}
 \begin{align*}
  \CC(IC(\S_0, 1))  & =  [\ol{T_{\S_0}^* V}].\\
   \CC (IC(\S_{11}, (4) )) & = [\ol{T_{\S_{11}}^* V}].\\
   \CC (IC(\S_2, 1^2)) & =  [\ol{T_{\S_2}^* V}].\\
  \CC (IC(\S_9, 1))& = [\ol{T_{\S_9}^* V}].\\
  \CC (IC(\S_5, 1)) & =  [\ol{T_{\S_5}^* V}].\\    
   \CC (IC(\S_6, 1)) & =  [\ol{T_{\S_6}^* V}].\\
  \CC (IC(\S_4, 1))& =  [\ol{T_{\S_4}^* V}].
  \end{align*} 
  \end{lemma}
 
 \begin{proof}
 The first identity is self-evident. 
 From   Table:2,  we observe that  $(\S_{11}, (4) ) = (^t(\S_0)_{FT}, ^t(\mathcal L_0)_{FT}).$
 Identity   (\ref{ident}) yields 
  \begin{align*}
   \CC (IC(\S_{11}, (4) )) &= [\ol{T_{\widehat{\S_0}}^* V}]\\
   &= [\ol{T_{\S_{11}}^* V}].
   \end{align*}
 
 Next, we consider $IC(\S_2, 1^2)$ and $IC(\S_9, 1) = IC(^t(\S_2)_{FT}, ^t(1^2)_{FT}).$
 Once again, we use  (\ref{ident}) and Table :1 to obtain:
 \begin{align*}
 \CC (IC(\S_2, 1^2)) &= [\ol{T_{\S_2}^* V}] + \chi^{mic}_{S_1}(IC(\S_2, 1^2))\;  [\ol{T_{\S_1}^* V}] + \chi^{mic}_{S_0}(IC(\S_2, 1^2))\;  [\ol{T_{\S_0}^* V}] ;\\
 \CC (IC(\S_9, 1)) &=  [\ol{T_{\widehat{\S_2}}^* V}] + \chi^{mic}_{S_1}(IC(\S_2, 1^2))\;  [\ol{T_{\widehat{\S_1}}^* V}] + \chi^{mic}_{S_0}(IC(\S_2, 1^2))\;  [\ol{T_{\widehat{\S_0}}^* V}]\,\\
  \CC (IC(\S_9, 1)) &=  [\ol{T_{\S_9}^* V}] + \chi^{mic}_{S_1}(IC(\S_2, 1^2))\;  [\ol{T_{\S_{10}}^* V}] + \chi^{mic}_{S_0}(IC(\S_2, 1^2))\;  [\ol{T_{\S_{11}}^* V}].
 \end{align*}
 
 Since neither $\S_{10}$ nor $\S_{11}$ are contained in the closure of $\S_9,$ we conclude that
 \[\chi^{mic}_{S_1}(IC(\S_2, 1^2)) =  \chi^{mic}_{S_0}(IC(\S_2, 1^2)) = 0.\] 
 Hence, $\CC (IC(\S_2, 1^2))$ and  $\CC (IC(\S_9, 1))$ are  as stated in the
 Lemma.  A similar argument computes the remaining $\CC.$ 
 \end{proof}
 \vskip .1in
 
 \begin{corollary}\label{c1}
 \[ c( \S_{10},\S_{11})  = 1.\]
 
 \end{corollary}
 
 \begin{proof}
 By Lemma \ref{l1}, $\chi^{mic}_{\S_{10}}(IC(\S_{11}, (4) )) =0$ where $\text{dim } \S_{10}= 11.$
 Hence,
 \begin{align*}
 0 & = \chi^{mic}_{\S_{10}}(IC(\S_{11}, (4) )) =  c(\S_{10}, \S_{10})\; \chi^{loc}_{\S_{10}}(IC(\S_{11}, (4)))
 + \; c(\S_{10}, \S_{11}) \; \chi^{loc}_{\S_{11}}(IC(\S_{11}, (4)))\\
 & = - \chi^{loc}_{\S_{10}}(IC(\S_{11}, (4)))
 + \; c(\S_{10}, \S_{11}) \; \chi^{loc}_{\S_{11}}(IC(\S_{11}, (4) )).
 \end{align*}
 
Here
\[ \chi^{loc}_{\S_{10}}(IC(\S_{11}, (4))) = \mathcal P_{(\S_{10}, 1),(\S_{11},4)}(1) + 
 \mathcal P_{(\S_{10}, 1^2),(\S_{11},4)}(1)  = 0 + 1, \text{ by \cite{ci08}.} \]
 Moreover, since  $\text{dim } \pi_{(4)} = 1,$ \cite{ci08} yields $\chi^{loc}_{\S_{11}}(IC(\S_{11}, (4))) =1.$
 We conclude that \[ 0 = -1 + c(\S_{10}, \S_{11}).\]
 \end{proof}
 
 \vskip .1in
 
 \begin{lemma}\label{l2}
 \begin{align*}
   \CC(IC(\S_1,1))  & =  [\ol{T_{\S_1}^* V}] + 3\;  [\ol{T_{\S_0}^* V}]. \\ 
  \CC (IC(\S_{11}, (31))) & = 3\; [\ol{T_{\S_{11}}^* V}]+[\ol{T_{\S_{10}}^* V}].\\
  \CC (IC(\S_2, 1)) & =  [\ol{T_{\S_2}^* V}]+[\ol{T_{\S_1}^* V}]+2\; [\ol{T_{\S_0}^* V}].\\
  \CC (IC(\S_{11}, (22))) & = 2\; [\ol{T_{\S_{11}}^* V}]+[\ol{T_{\S_{10}}^* V}]+[\ol{T_{\S_{9}}^* V}].\\
  \end{align*}
  \end{lemma}

  \begin{proof}
  We first consider $\CC(IC(\S_1,1)$ and $\CC(IC(^t(\S_1)_{FT},1) = \CC (IC(\S_{11}, (31)).$
  We have
  \begin{align*}
   \CC(IC(\S_1,1)) & =  [\ol{T_{\S_1}^* V}] + \chi^{mic}_{\S_0}(IC(\S_1,1)) \;  [\ol{T_{\S_0}^* V}] \\ 
  \CC (IC(\S_{11}, (31))) & =   [\ol{T_{\S_{10}}^* V}] + \chi^{mic}_{\S_0}(IC(\S_1,1)) \; [\ol{T_{\S_{11}}^* V}],\\
  \end{align*}
  where 
  \begin{align*}
  \chi^{mic}_{\S_0}(IC(\S_1,1)) & = \chi^{mic}_{\S_{11}}(IC(\S_{11}, (31)))  = c(\S_{11}, \S_{11})  \chi^{loc}_{\S_{11}}(IC(\S_{11}, (31)))\\
   &= \chi^{loc}_{\S_{11}}(IC(\S_{11}, (31))), 
  \text { as dim $\S_{11} = 12$}\\
   & = \text{dim } \pi_{(31)} = 3,
  \text{ by \cite{ci08}. }\\
  \end{align*}
 We have computed the first two characteristic cycles listed in Lemma \ref{l2}.
 \vskip .2in

 We use a similar argument to obtain
  \begin{align*}
   \CC (IC(\S_2, 1)) & =  [\ol{T_{\S_2}^* V}]+ \chi^{mic}_{\S_1}(IC(\S_2,1)) \;   [\ol{T_{\S_1}^* V}]+ 2\; [\ol{T_{\S_0}^* V}],\\
  \CC (IC(\S_{11}, (22))) & =   [\ol{T_{\S_{9}}^* V}] +   \chi^{mic}_{\S_1}(IC(\S_2,1)) \;   [\ol{T_{\S_{10}}^* V}]  + 2\; [\ol{T_{\S_{11}}^* V}], \text{ where}
  \end{align*}
  
   \begin{align*}
   \chi^{mic}_{\S_1}(IC(\S_2,1)) & =  \chi^{mic}_{\S_{10}}(IC(\S_{11}, (22))) \\
   & = c(\S_{10}, \S_{10})\;  \chi^{loc}_{\S_{10}}(IC(\S_{11}, (22))) + c(\S_{10}, \S_{11})\; \chi^{loc}_{\S_{11}}(IC(\S_{11}, (22)))\\
   & = - \; \chi^{loc}_{\S_{10}}(IC(\S_{11}, (22))) + \chi^{loc}_{\S_{11}}(IC(\S_{11}, (22))), \text{ by  Corollary \ref{c1}.}
  \end{align*}
   We use the tables of Kazhdan-Lusztig polynomials in \cite{ci08}  to compute
   $ \chi^{loc}_{\S_{10}}(IC(\S_{11}, (22))) = 1$ and  $\chi^{loc}_{\S_{11}}(IC(\S_{11}, (22))) = 2.$
   We get,
    $\chi^{mic}_{\S_1}(IC(\S_2,1))  = 1.$
   \end{proof}
   
  \vskip .2in
  
  \begin{corollary}\label{c2}
  \begin{align*}
  c(\S_8,\S_{11}) &= 1,  c(\S_9, \S_{11}) =1\\
  c(\S_8,\S_{10}) &= -2,  c(\S_9, \S_{10}) = -2.
  \end{align*}
  \end{corollary}
  
  \begin{proof}
  Lemma \ref{l1} and Lemma \ref{l2} yield two  systems of linear equations:
  \begin{align*}
  \chi^{mic}_{\S_{8}}(IC(\S_{11}, (4))) &= \chi^{mic}_{\S_{8}}(IC(\S_{11}, (31)))  =0, \\
 \chi^{mic}_{\S_{9}}(IC(\S_{11}, (4)))& =  \chi^{mic}_{\S_{9}}(IC(\S_{11}, (31))) =0.
 \end{align*}
 
 The values for $c(\S_8,\S_{10}),  c(\S_8, \S_{11})$ are the solutions of the first system of equations.
  The values for $c(\S_9,\S_{10}),  c(\S_9, \S_{11})$ are the solutions of the second system of equations.
  
  Using, once again the tables in \cite{ci08}, the first linear system of equations reads
  \begin{align*}
  0 & = 1 + 2 \;c(\S_8, \S_{10}) + 3\; c(\S_8,\S_{11})\\
  0 & = 1 + c(\S_8, \S_{10}) +  c(\S_8,\S_{11}).
  \end{align*}
  \end{proof}

   \vskip .2in
   
   \begin{lemma}\label{l3}
   \begin{align*}
  \CC (IC(\S_5, 1^2))& =  [\ol{T_{\S_5}^* V}]+ [\ol{T_{\S_3}^* V}]+[\ol{T_{\S_2}^* V}]+2[\ol{T_{\S_1}^* V}]+3[\ol{T_{\S_0}^* V}], \\
 \CC (IC(\S_{11}(211)) & = 3[\ol{T_{\S_{11}}^* V}]+2[\ol{T_{\S_{10}}^* V}]+[\ol{T_{\S_{9}}^* V}]+[\ol{T_{\S_{8}}^* X}]+[\ol{T_{\S_{6}}^* V}]\\
  \CC (IC(\S_3, 1)) & =  [\ol{T_{\S_3}^* V}]+[\ol{T_{\S_1}^* V}].\\
\CC (IC(\S_{10}, 1)) & =  [\ol{T_{\S_{10}}^* V}]+ [\ol{T_{\S_8}^* V}].
\end{align*}
\end{lemma}
  
 \begin{proof}
  
  We compute $\CC(IC(\S_3, 1))$ and $\CC(IC(^t(\S_3)_{FT}, 1)) = \CC(IC(\S_{10}, 1)).$
  The characteristic cycles,  $\CC (IC(\S_5, 1^2))$ and $\CC (IC(\S_{11}(211))$ are computed in a similar manner.
  \vskip .2in
  
 It is clear  that
  \begin{align*}
  \chi^{mic}_{\S_0}(IC(\S_3,1)) & = \chi^{mic}_{\S_{11}}(IC(\S_{10},1)) = 0.\\
  \chi^{mic}_{\S_1}(IC(\S_3,1)) & = \chi^{mic}_{\S_{10}}(IC(\S_{10},1)) = c(\S_{10}, \S_{10})\; \chi^{loc}_{\S_{10}}(IC(\S_{10},1)) = 1.
  \end{align*}
  
  We must compute
  \begin{align*}
   \chi^{mic}_{\S_2}(IC(\S_3,1))  &=  \chi^{mic}_{\S_{9}}(IC(\S_{10},1))\\
   &=  c(\S_9, \S_9) \; \chi^{loc}_{\S_{9}}(IC(\S_{10},1)) + c(\S_9,  \S_{10}) \; \chi^{loc}_{\S_{10}}(IC(\S_{10},1)) \\
  & =   \chi^{loc}_{\S_{9}}(IC(\S_{10},1)) -  c(\S_9,  \S_{10})\\
  & =   \chi^{loc}_{\S_{9}}(IC(\S_{10},1)) + 2, \text{ by Corollary \ref{c2}}. 
  \end{align*}
  
   This time, the third table on \cite[page 3]{ci08} lists the values at $q = 1$ of the
   Kazhdan-Lusztig polynomials $\mathcal P_{(\S_9, 1), (\S_{10},1)} (1) =  \mathcal P_{(\S_9, (1)^2), (\S_{10},1)}(1) = 1.$
 We deduce that $\chi^{loc}_{\S_{9}}(IC(\S_{10},1)) = (-1)^{\text{dim}\S_{10}}\;  2 = - 2.$ 
 That is,
 $ \chi^{mic}_{\S_2}(IC(\S_3,1)) =  \chi^{loc}_{\S_{9}}(IC(\S_{10},1)) + 2 = -2 + 2 = 0.$
 \end{proof}
 \vskip .2in
 
 \begin{corollary}\label{c3}
 \begin{alignat*}{6}
   c(\S_5, \S_{11}) &= 1, &&& c(\S_7, \S_{11}) &= 1,&&& c(\S_4, \S_{10}) &= - 3 \; c(\S_4, \S_{11}) \\
   c(\S_5, \S_{10}) &= - 3, &&& c(\S_7, \S_{10}) &= -3, &&&{}                     \\
   c(\S_5, \S_{9}) &=  0, &&& c(\S_7, \S_{9}) &= 1, &&& c(\S_4, \S_{9}) &= 1+ \; c(\S_4, \S_{11}) \\
   c(\S_5, \S_8) & = 3, &&& c(\S_7,\S_8) & = 2,&&& c(\S_4, \S_{8}) &= 2 \; c(\S_4, \S_{11})\\
    c(\S_5, \S_{7}) &=  -2, &&& c(\S_7, \S_{7}) &= -1&&& c(\S_4, \S_{7}) &=  - \; c(\S_4, \S_{11}). 
\end{alignat*}
\end{corollary}

\begin{proof}
The first column is the solution to the system of equations
\begin{align*}
\chi^{mic}_{\S_5}(IC(\S_{11},(4)) &= \chi^{mic}_{\S_5}(IC(\S_{11},(31)) =\chi^{mic}_{\S_5}(IC(\S_{11},(22))= \chi^{mic}_{\S_5}(IC(\S_{11},(211)) =0\\
\chi^{mic}_{\S_5}(IC(\S_{9},(1)) &=  \chi^{mic}_{\S_5}(IC(\S_{9},(1^2)) = 0.
\end{align*}
Note that all  the listed equations but  $\chi^{mic}_{\S_5}(IC(\S_{9},(1^2)) = 0$  follow from the computations
included in Lemma \ref{l1}, Lemma \ref{l2} and Lemma \ref{l3}.
To justify the extra equation recall that $IC(\S_{9},(1^2)) = (IC(^t(\S_{9})_{FT},^t(1^2)_{FT})$ and $\widehat{\S_5} = \S_6.$
This implies that  either both $\overline{T^*_{\S_5}V}$ and $\overline{T^*_{\S_6}V}$  contribute to
to $\CC(IC(\S_{9},(1^2)),$  or neither of the two conormal bundles occur in $\CC(IC(\S_{9},(1^2)).$
Since $\S_6$ is not contained in the closure if $\S_9,$ we conclude that $ \chi^{mic}_{\S_5}(IC(\S_{9},(1^2)) = 0.$

The second column is the solution to the system of equations
\begin{align*}
\chi^{mic}_{\S_7}(IC(\S_{11},(4)) &= \chi^{mic}_{\S_7}(IC(\S_{11},(31)) =\chi^{mic}_{\S_7}(IC(\S_{11},(22))= \chi^{mic}_{\S_7}(IC(\S_{11},(211)) =0.
\end{align*}

The third column is the solution to the system of equations
\begin{align*}
\chi^{mic}_{\S_4}(IC(\S_{11},(4)) &= \chi^{mic}_{\S_4}(IC(\S_{11},(31)) =\chi^{mic}_{\S_4}(IC(\S_{11},(22))= \chi^{mic}_{\S_4}(IC(\S_{11},(211)) =0\\
\chi^{mic}_{\S_4}(IC(\S_{9},(1)) &=  \chi^{mic}_{\S_4}(IC(\S_{10},(1)) = 0.
\end{align*}
\end{proof}

\vskip .2in
\begin{lemma}\label{l4}
There is constant $c \geq 2$ such that
\begin{align*}
 \CC (IC(\S_7, 1))& =  [\ol{T_{\S_{7}}^* V}]+(c+1)[\ol{T_{\S_{4}}^* V}]+[\ol{T_{\S_{3}}^* V}]+2[\ol{T_{\S_{2}}^* V}]+[\ol{T_{\S_{1}}^* V}],\\
 \CC (IC(\S_8, 1))& = [\ol{T_{\S_8}^* V}]+ [\ol{T_{\S_7}^* V}]+2[\ol{T_{\S_6}^* V}]+2[\ol{T_{\S_5}^* V}]+(c-2)[\ol{T_{\S_4}^* V}]+[\ol{T_{\S_3}^* V}],\\
 \CC (IC(\S_9, 1^2)) & =  [\ol{T_{\S_9}^* V}]+[\ol{T_{\S_7}^*V}]+c[\ol{T_{\S_4}^* V}]+[\ol{T_{\S_2}^* V}],\\
 \CC (IC(\S_{10}, 1^2)) & = [\ol{T_{\S_{10}}^* V}]+2[\ol{T_{\S_{9}}^* V}]+[\ol{T_{\S_{8}}^* V}]+[\ol{T_{\S_{7}}^* V}]+(c+1)[\ol{T_{\S_{4}}^* V}].
\end{align*}
\end{lemma}

\begin{proof}
This lemma is proved following the same strategy used to settle other lemmas in this appendix.
For the sake of completeness we  include the computation of  
 $\CC (IC(\S_8, 1))$ 
Since $IC(\S_8, 1) = IC(^t(\S_8)_{FT}, ^t(1)_{FT}),$ we have
\begin{align*}
 \CC (IC(\S_8, 1))  &= [\ol{T_{\S_8}^* V}]+\chi^{mic}_{\S_7}((IC(\S_8, 1))  [\ol{T_{\S_7}^* V}]+ 
 \chi^{mic}_{\S_5}((IC(\S_8, 1)) [\ol{T_{\S_6}^* V}]+\\ 
 &+\chi^{mic}_{\S_5}((IC(\S_8, 1)) [\ol{T_{\S_5}^* V}]+
 \chi^{mic}_{\S_4}((IC(\S_8, 1)) [\ol{T_{\S_4}^* V}]+    [\ol{T_{\S_3}^* V}].
 \end{align*}

Once again, we use  the tables of Kazhdan-Lusztig polynomials in \cite{ci08} to compute local multiplicities.
We obtain:

\begin{align*}
 \chi^{mic}_{\S_7}((IC(\S_8, 1)) & = c(\S_7, \S_7) + c(\S_7, \S_8) = -1 + c(\S_7, \S_8)\\
  \chi^{mic}_{\S_5}((IC(\S_8, 1)) & =  c(\S_5, \S_5) +  c(\S_5, \S_7) + c(\S_5, \S_8) = 1 +
   c(\S_5, \S_7) + c(\S_5, \S_8)\\
   \chi^{mic}_{\S_4}((IC(\S_8, 1))& = 2 c(\S_4, \S_4) + c(\S_4, \S_7) + c(\S_4, \S_8)
   = -2 + c(\S_4, \S_7) + c(\S_4, \S_8).
   \end{align*}
   
 We complete the computation using the values of the  various $c(\S, \S')$   listed in Corollary \ref{c3}. 
 In particular, $ \chi^{mic}_{\S_4}((IC(\S_8, 1))  = c(\S_4,\S_{11}) - 2 = c -2.$

\end{proof}


\begin{thebibliography}{CMBO}

\bibitem[ABV92]{abv} J.~Adams, D.~Barbasch, and D.A. ~Vogan, Jr., {\em
The Langlands Classification and Irreducible Characters for Real
Reductive Groups}, Progress in Math, Birkh\"auser (Boston), {\bf 104}(1992).
\vskip .1in

\bibitem[Art89]{Art89} J.~ Arthur, {\em Unipotent automorphic representations: conjectures},
in {\em Unipotent orbits and representations.} {\bf 171-172}. Ast\'erisque. Soci\'et\'e Math\'ematique de France (1989).
\vskip .1in

\bibitem[Au95]{Au95} A.~Aubert, {\em Dualit\'e dans groupe de Grothendieck de la cat\'egorie
des repres\'entations  lisses de longueur finie de'un groupe r\'eductif p-adique}, Trans. Amer. Math. Soc.,
{\bf 347} (1995), no.~6, 2179-2189.
\vskip .1in

\bibitem[Ci08]{ci08} D.~Ciubotaru, {\em Multiplicity matrices for the affine graded Hecke algebra},
J. of Algebra {\bf Vol. 320} (2008),  no.~11, 3950-3983.
\vskip .1in

\bibitem[Ci22]{ci22} D.~Ciubotaru, {\em Weyl groups, the Dirac inequality, and isolated unitary unramified representations},
Indag. Math. {\bf Vol. 33} (2022),  no.~1, 1-23.
\vskip .1in



\bibitem[CMBO24]{cmbo} D.~Ciubotaru, L.~Mason-Brown, and E.~Okada, 
{\em Some unipotent Arthur packets for reductive p-adic groups}, Int. Math. Res. Not. (2024), no.~ 9, 7502-7525.
\vskip .1in

\bibitem[CMcG93]{cmcg} D.~Collingwood and W.~McGovern, {\em Nilpotent orbits in semisimple Lie algebras},
Van Nostrand Reinhold Math. Series (New York) (1993).
\vskip .1in

\bibitem[CFMMX]{cfmmx} C.~Cunningham, A.~Fiori, A.~Moussaoui, J.~Mracek, and B.~Xu,
{\em Arthur packets for p-adic groups by way of microlocal vanishing cycles of perverse sheaves, with examples},
Mem. Amer. Math. Soc., {\bf 276} (2022), no.~1353.
\vskip .1in

\bibitem[EM97]{em} E.~Evens and I.~Mirkovi\'c, {\em Fourier transform and the Iwahori-Matsumoto  involution},
Duke Math., {\bf 86} (1997), no.~5, 435-464.
\vskip .1in

\bibitem[Gi86]{gz} V.~Ginzburg, {\em Characteristic varieties and vanishing cycles}, Invent. Math. {\bf 84} (1986), 327-402.
\vskip .1in

\bibitem[HK84]{hk} R.~Hotta and M.~Kashiwara, {\em The invariant holonomic system on a semisimple Lie algebra},
Invent. Math. {\bf 75} (1984), 327-358.
\vskip .1in

\bibitem[Ka83]{ka} M.~Kashiwara, {\em Systems of Microdifferential equation}, Progress In Math. {\bf 34},
Birkh\"auser, Boston-Basel-Stuttgart, 1983.
\vskip .1in

\bibitem[KS90]{ks} M.~Kashiwara and P.~Schapira, {\em Sheaves on Manifolds}, Grundlehren Math. Wiss. {\bf 292},
Springer-Verlag, Berlin, 1990. 
\vskip .1in

\bibitem[KM79]{km} T.~Kimura and M.~Muro, {\em On some series of regular prehomogeneous vector spaces},
Proc. Japan Acad. {\bf 55} Ser. A (1979), 384-389.

\vskip .1in

\bibitem[Kna97]{Kn} A.~Knapp, {\em Introduction to the Langlands program}, Representation theory and automorphic forms
(Edinburg, 1996) {\bf Vol. 61} Proc. Symp. Pure Math. (1997), 245-302.
\vskip .1in

\bibitem[LL25]{ll} J. Li and A.C.~ L\H{o}rincz, {\em On the Fourier transform of D-modules on prehomogeneous spaces}, preprint (2025).
\vskip .1in

\bibitem[L\H{o}19]{lo} A.C.~ L\H{o}rincz, {\em Holonomic functions and prehomogeneous spaces}, Selecta Math. {\bf 29} (2023), no.~5, Paper No.~29, 48.

\vskip .1in
\bibitem[LW19]{lw} A.C.~ L\H{o}rincz and U.~Walther, {\em On categories of equivariant $D$-modules}, Adv. Math. {\bf 351} (2019), 429--478.
 \vskip .1in
 
 \bibitem[LY25]{ly} A.C.~ L\H{o}rincz and Ruijie Yang, {\em Filtrations of $D$-modules along semi-invariant functions}, \href{https://arxiv.org/abs/2504.19383}{arXiv:2504.19383}.
 \vskip .1in 

 
\bibitem[Lus95]{Lus95} G.~Lusztig, {\em Classification of unipotent representations of simple p-adic groups},
Int. Math. Res. Notices {\bf 11} (1995), 517-589.
\vskip .1in

\bibitem[Pja75]{Pja75} V.~S.~ Pjasecki\u{i}, {\em Linear Lie groups that act with  a finite number of orbits},
Funkcional Anal. i Prilo\v{z}en  {\bf 9} (1975), no.~4, 85-86.
\vskip .1in


\bibitem[SKKO80]{micro} M. Sato, M. Kashiwara, T. Kimura, and T. Oshima, {\em Micro-local analysis of prehomogeneous vector spaces}, Invent. Math. {\bf 62} (1980), 117-179.
\vskip .1in

\bibitem[Sp82]{sp} N.~Spaltenstein, {\em Classes Unipotentes et Sous-groupes de Borel}, Springer
Heidelberg, 1982.
\vskip .1in

 \bibitem[Ta79]{ta} J.~Tate, {\em Number theoretical background}, in {\em Automorphic Representations and
 L-Functions}, Proc. Sympos. Math., {\bf 33} (1979), 3-26.


\vskip .1in
\bibitem[Vo93]{Vo93}
D.~A.~Vogan, Jr., 
{\em The local Langlands conjecture}, in 
{\em Representation theory of groups and algebras}, 
305--379, Contemp.~Math., {\bf 145}, Amer~ Math.~Soc.(Providence, RI), 1993.


\end{thebibliography}
\end{document}